\documentclass[12pt,a4paper]{article}

\usepackage{amssymb, amsmath, amsthm}

\usepackage[cm]{fullpage}
\usepackage[english]{babel}
\usepackage[cp1250]{inputenc}
\usepackage[T1]{fontenc}
\usepackage{booktabs}

\usepackage{hyperref}

\usepackage[pdftex]{graphicx}
\usepackage{tikz}
\usetikzlibrary{decorations.pathmorphing,patterns}

\newtheorem{thm}{Theorem}
\newtheorem{lem}{Lemma}

\newtheorem{cor}{Corollary} 
\newtheorem{prob}{Problem}

\title{Solution and asymptotic analysis of a boundary value problem in the spring-mass model of running}
\author{\L ukasz P\l ociniczak\thanks{Faculty of Pure and Applied Mathematics, Wroc{\l}aw University of Science and Technology, Wyb. Wyspia{\'n}skiego 27, 50-370 Wroc{\l}aw, Poland}$\;^,$\footnote{Email: lukasz.plociniczak@pwr.edu.pl}, Zofia Wr{\'o}blewska$^*$}
\date{}

\begin{document}
	\maketitle
	
	\begin{abstract}	
		We consider the classic spring-mass model of running which is built upon an inverted elastic pendulum. In a natural way, there arises an interesting boundary value problem for the governing system of two nonlinear ordinary differential equations. It requires us to choose the stiffness to ascertain that after a complete step, the spring returns to its equilibrium position.  Motivated by numerical calculations and real data we conduct a rigorous asymptotic analysis in terms of the Poicar\'e-Lindstedt series. The perturbation expansion is furnished by an interplay of two time scales what has an significant impact on the order of convergence. Further, we use these asymptotic estimates to prove that there exists a unique solution to the aforementioned boundary value problem and provide an approximation to the sought stiffness. Our results rigorously explain several observations made by other researchers concerning the dependence of stiffness on the initial angle of the stride and its velocity. The theory is illustrated with a number of numerical calculations. \\
		
		
	\end{abstract}
	
	\section{Introduction}
	Running is the fundamental way of rapid legged locomotion for terrestrial animals and due to its naturalness and everyday occurrence, it seems that there is nothing unusual in it. However, this way of movement requires a complex and accurate collaboration of neural, motor, and muscular systems with respect to the changing terrain \cite{Dic10}. The usual and common distinction between walking and running is that the latter contains an aerial phase during which the animal has no contact with the ground. This working definition is sufficient for us, however as research shows, it is too narrow to include certain animals or conditions of locomotion (see \cite{Cav76} for a broader classification based on the energetic concepts). Running is not just walking with a higher speed and there is a remarkable transition of one mean of motion to the other \cite{Ale89}. Legged locomotion of various animals has been investigated vigorously through the decades and it merges biology, engineering and mathematics into one successful endeavour. This topic was investigated by Aristotle \cite{Nus85} while the first biomechanical treatment was given by a seventeenth century Italian physiologist and mathematician Giovanni Borelli \cite{Bor43,Maq89}. The Reader can find interesting modern surveys concerning scientific accounts of locomotion in \cite{Bie18,Dic10,Tak17}.
	
	For humans, running gains another important meaning apart of being simple way of locomotion - namely, the sports. As a primal form of movement, running accompanies man from the very beginning. It is difficult not to agree with the fact that running is probably the simplest and the most natural sport that exist \cite{Ale91}. This is why it forms a basis for many other disciplines. In sport science it is quite common to analyse and to attempt to describe some movements that are specific to the considered discipline \cite{Fer99}. Many investigations lead to a better understanding of human performance during races of different lengths \cite{Dan78,And96,Wil87} which, in turn, provided better insights on improved training methods \cite{Dan13}. Mathematically, competitive running was described by Keller in his variational model \cite{Kel73} based on a physiological findings of Hill \cite{Hil25}. Some further generalizations and analysis are given in \cite{Pri93,Beh93,Tib05}. 
	
	Mathematical modelling is an important part of the biomechanics. In this paper we are concerned with the so-called spring-mass model of running which is based on an inverted elastic pendulum (see the seminal papers \cite{Bli89,Mc90}). However, some earlier attempts accurately described walking with a similar construction utilizing inverted pendulum \cite{Moh80}. These first investigations lead to a proliferation of interesting concepts, models, and methods that helped to design walking robots \cite{Col05}. Further generalizations of the spring-mass model are based on including additional legs, dampers and segments \cite{Gey06,Sie82,Sri08,Ale92,Rum08}. The two-legged version has an interesting bifurcation structure \cite{Mer15}. A very thorough review of models of legged locomotion is given in \cite{Hol06}.
	
	In this paper we revisit the conceptual spring-mass model. Our focus is to solve the naturally arising boundary-value problem for the spring stiffness via the asymptotic analysis. However, being essentially an elastic pendulum the mathematical description of the system consists of two nonlinear ordinary differential equations which possess a rich geometrical structure \cite{Geo99,Hol02} and chaotic behaviour \cite{Cue92,Ala06}. 
	
	The paper is structured as follows. In Section 2 we state derive the model and state the main boundary value problem. By a numerical calculations we motivate that a perturbative expansion with respect to the large spring stiffness is relevant for the solution of the problem. Section 3 contains asymptotic analysis of main equations with the use of Poincar\'e-Lindstedt series. The material is divided into two parts: one gives heuristic derivation of the perturbation solutions while the other rigorously justifies them. In Section 4 we solve the initially stated boundary value problem and provide an approximation for its solution. We close the paper with several numerical calculations verifying our theory. 
	
	\section{Model statement and motivation}
	The spring-mass model of running assumes that each leg can be described as an inverted elastic pendulum. For the grounded phase of the jump, we assume the situation depicted on Fig. \ref{fig:Schematic}. 
	
	\begin{figure}
		\centering
		\begin{tikzpicture}[scale = 2]
		\draw (0:0) -- (120:2);
		\draw [dotted] (0:0) -- (60:2);
		\draw [dashed] (0,0) -- (0,5);
		\draw [dashed, <->] (210:0.4) -- (125:5) node[midway, below, sloped] {$L$};
		
		\draw[decoration={aspect = 0.3, segment length = 3mm, amplitude = 3mm, coil}, decorate] (120:2) -- (120:4.6);
		\draw[dotted, decoration={aspect = 0.3, segment length = 3mm, amplitude = 3mm, coil}, decorate] (60:2) -- (60:4.6);
		
		\draw (120:5) circle [radius = 0.4] node {$m$};
		\draw [dotted] (60:5) circle [radius = 0.4];
		
		\fill [pattern = north east lines] (-3,-0.25) rectangle (3, 0);
		\draw [very thick] (-3,0) -- (3,0);
		
		\draw (90:3) arc (90:120:3);
		\draw (105:1.5) node {$\theta$};
		
		\draw [dotted] (120:5) .. controls (90:3.5) .. (60:5);
		
		\end{tikzpicture}
		\caption{A schematic of the main model.}
		\label{fig:Schematic}
	\end{figure}
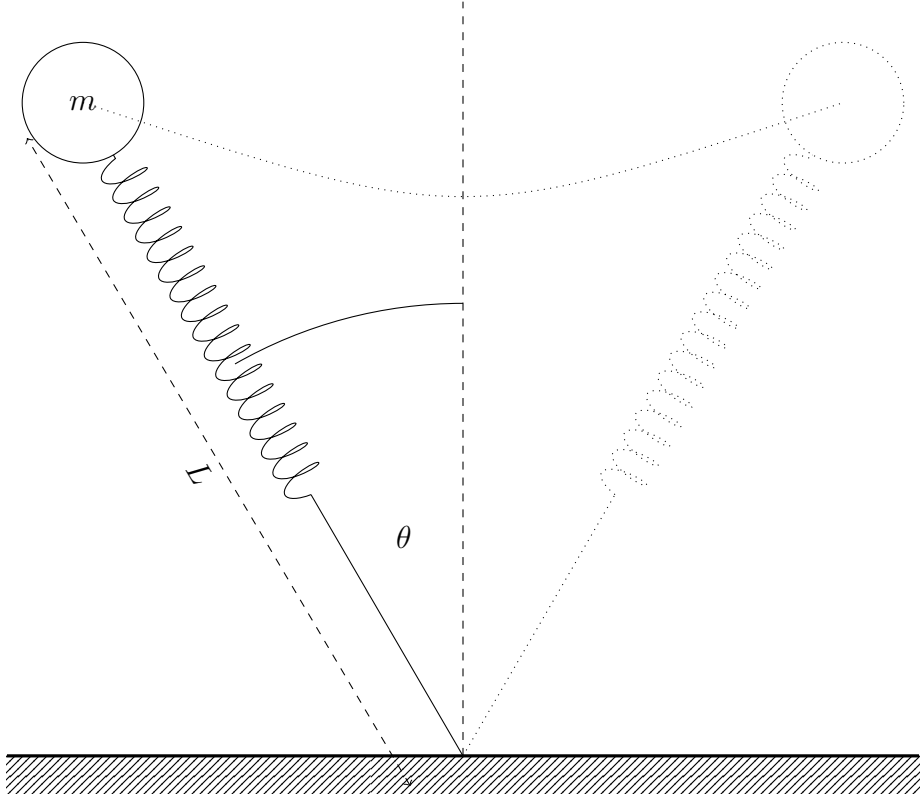
	
	Let $(x(t),y(t))$ denote the Cartesian coordinates of the point mass $m$. Balancing respective components of gravity and stiffness we can write
	\begin{equation}
	\left\{
	\begin{array}{l}
	m\dfrac{d^2 x}{dt^2} = k \left(l_0-\sqrt{x^2+y^2}\right)\sin\theta, \vspace{6pt}\\
	m\dfrac{d^2 y}{dt^2} = k \left(l_0-\sqrt{x^2+y^2}\right)\cos\theta- m g,
	\end{array}
	\right.
	\end{equation}
	where $l_0$ is the equilibrium length of the spring and $k$ is the stiffness. Plugging in the Cartesian formulas for the angle we obtain
	\begin{equation}
	\label{eqn:MainEqCartesian}
	\left\{
	\begin{array}{l}
	\dfrac{d^2 x}{dt^2} = \dfrac{k }{m}x\left(\dfrac{l_0}{\sqrt{x^2+y^2}}-1\right), \vspace{6pt}\\
	\dfrac{d^2 y}{dt^2} = \dfrac{k}{m}y\left(\dfrac{l_0}{\sqrt{x^2+y^2}}-1\right)- g.
	\end{array}
	\right.
	\end{equation}
	The initial conditions are the following
	\begin{equation}
	\label{eqn:ICCartesian}
	x(0) = -l_0 \sin\alpha, \quad \frac{dx}{dt}(0) = u, \quad y(0) = l_0 \cos\alpha, \quad \frac{dy}{dt}(0) = v,
	\end{equation}
	where $u$ and $v$ are, respectively, horizontal and vertical velocities. It is very convenient to cast the governing equations (\ref{eqn:MainEqCartesian}) into nondimensional polar form. To this end, we naturally scale $x$ and $y$ with respect to $l_0$ and choose the pendulum time scale $\sqrt{g/l_0}$. Moreover, we introduce the nondimensional spring length $L = \sqrt{x^2+y^2}/l_0$ and the polar angle $\theta$. Eventually, the polar form of (\ref{eqn:MainEqCartesian}) is the following
	\begin{equation}
	\label{eqn:MainEq}
	\left\{
	\begin{array}{l}
	L\dfrac{d^2 \theta}{dt^2} + 2 \dfrac{dL}{dt} \dfrac{d\theta}{dt} = \sin\theta, \vspace{6pt}\\
	\dfrac{d^2 L}{dt^2} - \left(\dfrac{d\theta}{dt}\right)^2 L = K \left(1-L\right)-\cos\theta,
	\end{array}
	\right.
	\end{equation}
	where the only nondimensional parameter (spring stiffness) present in the equations is given by
	\begin{equation}
	K = \frac{k l_0}{m g}.
	\end{equation}
	The initial conditions for the polar coordinate system have the form
	\begin{equation}
	\label{eqn:ICPolar}
	\theta(0) = -\alpha, \quad \frac{d\theta}{dt}(0) = \theta_d, \quad L(0) = 1, \quad \frac{dL}{dt}(0) = - L_d,
	\end{equation}
	with
	\begin{equation}
	\label{eqn:ThetaLdUV}
	\theta_d = U \cos\alpha-V \sin\alpha, \quad L_d = U \sin\alpha + V \cos\alpha,
	\end{equation}
	where we have defined the horizontal and vertical Froude numbers
	\begin{equation}
	U = \frac{u}{\sqrt{g l_0}}, \quad V = \frac{v}{\sqrt{g l_0}}.
	\end{equation}
	In Tab. \ref{tab:Parameters} we have collected all nondimensional parameters appearing in the model. Notice that usually $U$ is of order of unity, while $V$ and $\alpha$ are small. 
	
	\begin{table}
		\centering
		\begin{tabular}{llr}
			\toprule
			Symbol & Description & Typical value \\ 
			\midrule
			$\alpha$ & Angle of attack & $0.2-0.8$\\
			$U$ & Horizontal Froude number & $0.8-2.6$\\
			$V$ & Vertical Froude number & $0.05-0.5$\\
			\bottomrule
		\end{tabular}
		\caption{Typical values of all appearing nondimensional physical parameters. Data based on \cite{Far93} and calculated for various animals.}
		\label{tab:Parameters}
	\end{table}
	
	By the standard theory of ordinary differential equations the system (\ref{eqn:MainEq}) with (\ref{eqn:ICPolar}) possesses a unique solution in the vicinity of $t=0$. On the other hand, from the point of view of applications a question of completely different nature is relevant. Since the inverted elastic pendulum models the forwardly hoping leg we are faced with a peculiar boundary value problem to solve.
	\begin{prob}
		\label{prob:main}
		Let $(\theta(t;K),L(t,K))$ be the solution of the system (\ref{eqn:MainEq}) with (\ref{eqn:ICPolar}). Find $K^*$ and the smallest time $t^*>0$ satisfying
		\begin{equation}
		\theta(t^*, K^*) = \alpha, \quad L(t^*,K^*) = 1.
		\end{equation} 
	\end{prob}
	This means that the spring stiffness $K$ has to be determined to ensure that during the first cycle the spring will return to the equilibrium length precisely at the time for which the pendulum travels to the angle $\alpha$. This represents the grounded phase of the step, i.e. the leg completes the full cycle before jumping into the aerial phase (see Fig. \ref{fig:Schematic}). 
	
	The above problem can easily be solved numerically using the shooting method as was done for example in \cite{Mc90}. To illustrate this, we apply a numerical solver based on the forth order Runge-Kutta method for solving the initial value problem (\ref{eqn:MainEq}) and (\ref{eqn:ICPolar}) with a given $K$. Then, the point $t^*$ is found such that $\theta(t^*) = \alpha$. Next, $L(t^*)$ is compared with $1$ and according to the difference the value of $K$ is corrected via the secant method for the next iteration. The loop continues until the required accuracy is attained. 
	
	Numerically found values of $K^*$ are depicted on Fig. \ref{fig:Stiffness} with respect to the initial angle of attack $\alpha$ for several values of $U$ and vice-versa. We see that in general $K^*$ is a moderately large parameter especially for small angles but also for realistic regime of constants (see Tab. \ref{tab:Parameters}). The dependence on $\alpha$ and $U$ is monotone and we can anticipate that for $U$ of orders of unity $K^*$ has a quadratic component of $U$. In the following sections we will prove these claims along with finding asymptotic expansions of $L$ and $\theta$ and proving existence and uniqueness for Problem \ref{prob:main}.
	
	\begin{figure}[!htbp]
		\centering
		\includegraphics[width=0.49\textwidth]{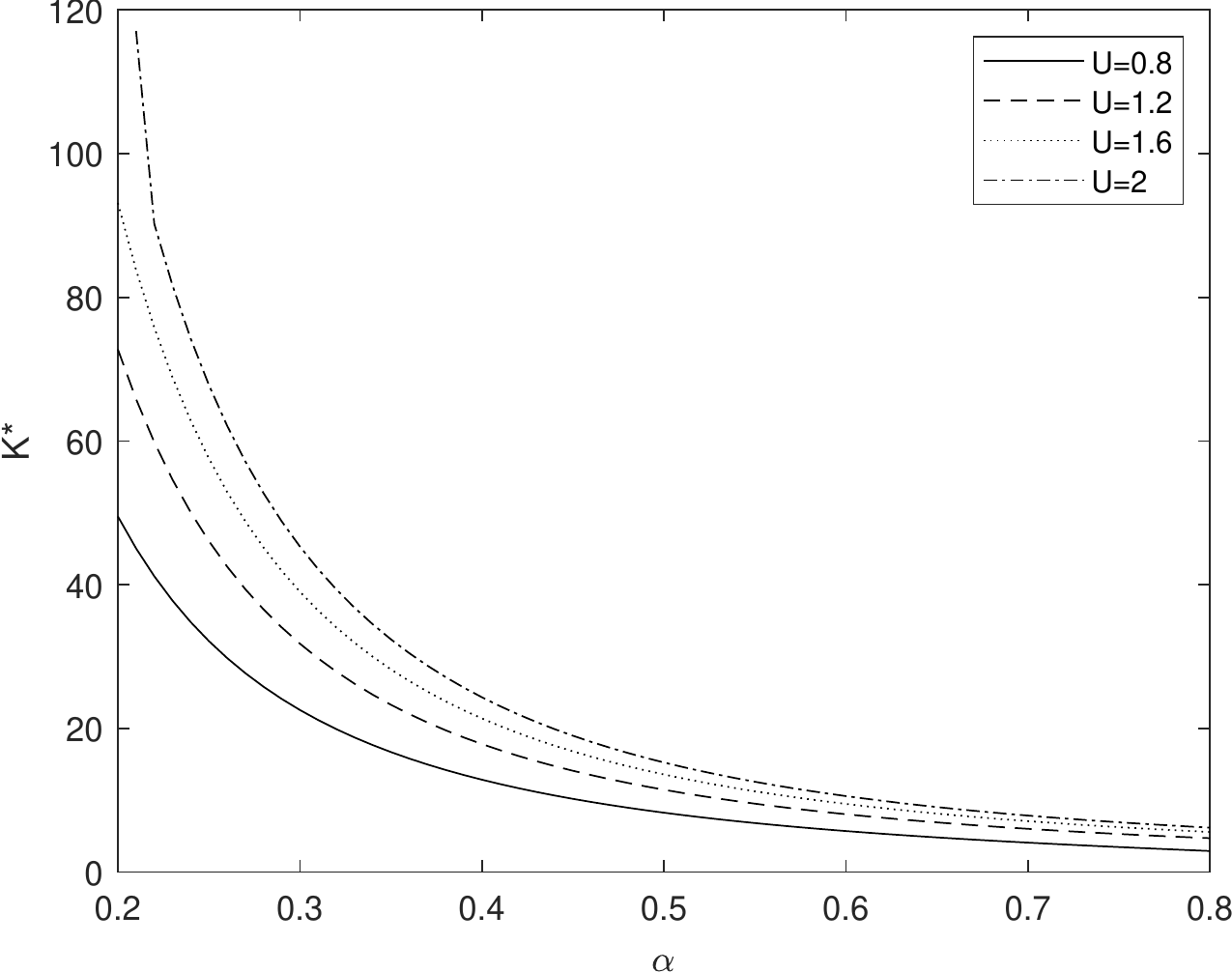}
		\includegraphics[width=0.49\textwidth]{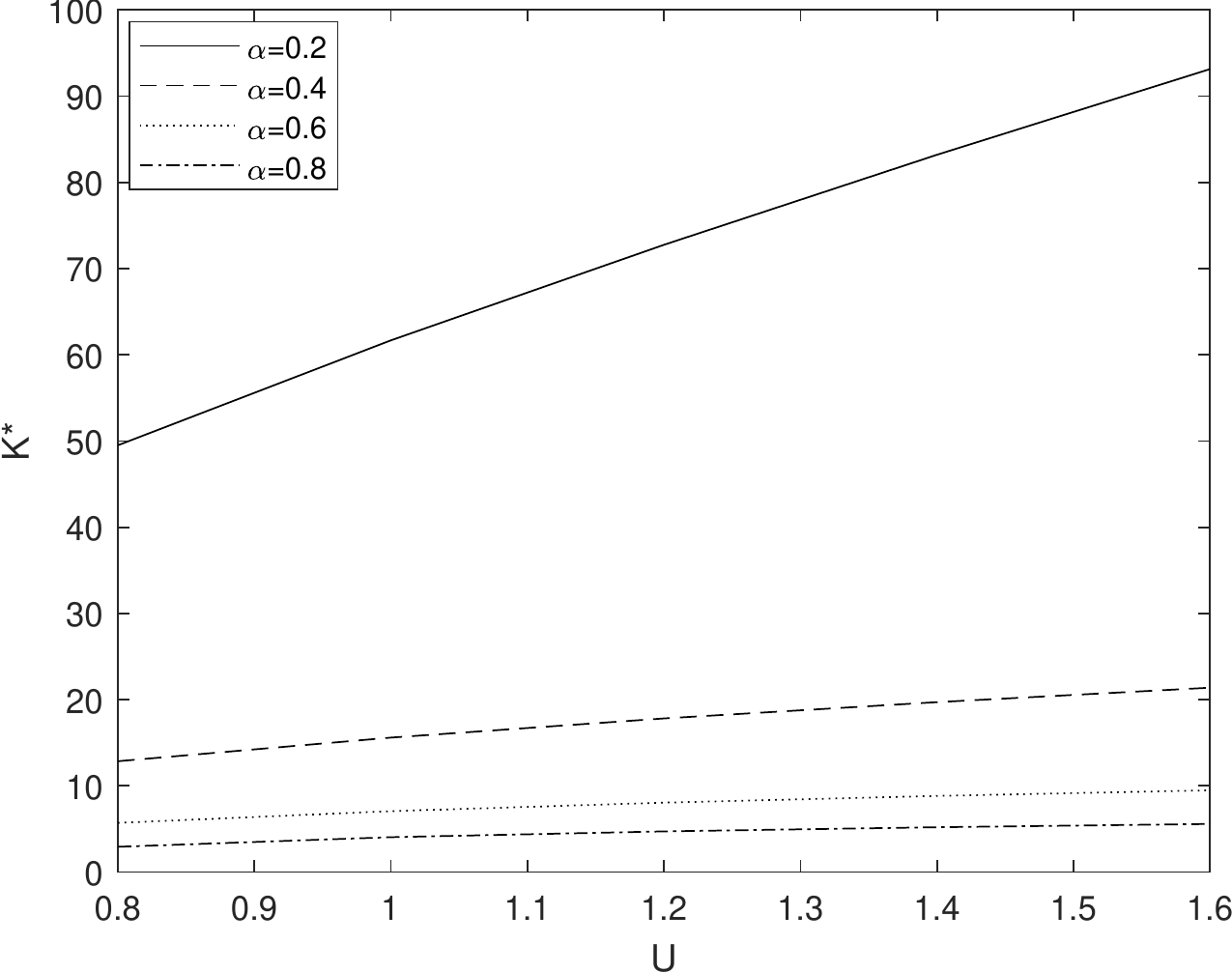}
		\caption{Numerical solution of Problem \ref{prob:main} for different sets of $\alpha$ and $U$ with a fixed value of $V=0.1$.}
		\label{fig:Stiffness}
	\end{figure}
	
	\section{Perturbation theory}
	
	In this section we provide an asymptotic analysis for solutions of the system (\ref{eqn:MainEq}) with (\ref{eqn:ICPolar}) for large $K$. First, to get a hint how the possible asymptotic solution may look like we proceed with the formal singular perturbation theory. Then, we provide rigorous proofs concerning the order of approximation. 
	
	\subsection{Formal expansions}
	
	The above numerical results (compare \cite{Mc90}) indicate that for a wide range of realistic initial conditions the appropriate stiffness being the solution of the Problem \ref{prob:main} is large. This suggests that we can gain a meaningful insight by expanding (\ref{eqn:MainEq}) for $K\rightarrow\infty$. To this end we will use the standard Poincar\'e-Lindstedt method (see \cite{Hol12}).
	
	We will see that solutions live on different time scales. Since $K$ is the factor of $L$ in the second equation in (\ref{eqn:MainEq}) we expect that the dynamics takes place on the fast $\sqrt{K} t$ scale. On the other hand, the equation for the angle indicates that the main time scale for $\theta$ is the slow one $t$. Furthermore, due to the coefficient $\left(d\theta/dt\right)^2$ the period of oscillations is modulated by the slower time $t$. However, as we will shortly see, the solution of Problem \ref{prob:main} occurs on the fast scale and it will be more convenient to expand both variables with respect to that. 
	
	First, we will use the Poincar\'e-Lindstedt series to find the asymptotic behaviour of $L$ and $\theta$ as $K\rightarrow\infty$. To simplify matters we set
	\begin{equation}
	\label{eqn:epsilon}
	\epsilon = \frac{1}{\sqrt{K}},
	\end{equation}
	hence we are looking for an expansion as $\epsilon\rightarrow 0$. Being lead by the above discussion we introduce the following fast time scales
	\begin{equation}
	\label{eqn:TimeScales}
	\tau = \frac{t}{\epsilon}, \quad \tau^+ = \omega(\epsilon)\tau = \left(1+\epsilon \omega_1 + \epsilon^2 \omega_2+...\right)\frac{t}{\epsilon}.
	\end{equation}
	Now, substituting $\tau^+$ from (\ref{eqn:TimeScales}) into (\ref{eqn:MainEq}) we obtain
	\begin{equation}
	\label{eqn:LTau+}
	\begin{split}
	\omega(\epsilon)^2 L''+\left(1-\omega(\epsilon)^2 (\theta')^2\right)L = 1 - \epsilon^2\cos\theta, \quad L\theta''+2 L' \theta' = \omega(\epsilon)^{-2}\epsilon^2 \sin\theta,
	\end{split}
	\end{equation}
	where now $L=L(\tau^+)$, $\theta=\theta(\tau^+)$, and prime denotes the derivative with respect to $\tau^+$. If we make the following formal asymptotic expansions
	\begin{equation}
	L = L_0 + \epsilon L_1 + \epsilon^2 L_2 + ..., \quad \theta = \theta_0+\epsilon\theta_1+\epsilon^2\theta_2 + ...,
	\end{equation}
	the equations become
	\begin{equation}
	\begin{split}
	\omega(\epsilon)^2 \left(L''_0 + \epsilon L''_1 + \epsilon^2 L''_2 + ...\right) &+\left(1 - \omega(\epsilon)^2\left(\theta'_0+\epsilon\theta'_1+\epsilon^2\theta_2'+...\right)^2\right)\left(L_0 + \epsilon L_1 + \epsilon^2 L_2 + ...\right) \\
	&= 1-\epsilon^2 \cos\theta,
	\end{split}
	\end{equation}
	and
	\begin{equation}
	\begin{split}
	\left(L_0 + \epsilon L_1 + \epsilon^2 L_2 + ...\right)\left(\theta_0''+\epsilon \theta_1'' + \epsilon^2\theta_2'' +...\right) + 2 &\left(L_0' + \epsilon L'_1 + \epsilon^2 L'_2 + ...\right)\left(\theta'_0+\epsilon\theta'_1+\epsilon^2\theta_2'+...\right)\\ 
	&= \omega(\epsilon)^{-2} \epsilon^2 \sin\left(\theta_0+\epsilon\theta_1+\epsilon^2\theta_2+...\right).
	\end{split}
	\end{equation}
	Next, collecting respective coefficients of $\epsilon^n$ in (\ref{eqn:LTau+}) we obtain the following chain of differential equations
	\begin{equation}
	\begin{array}{lll}
	\epsilon^0:& L''_0+\left(1-(\theta_0')^2\right)L_0 = 1, & L_0(0)=1, \quad L'_0(0) = 0, \\
	\epsilon^1:& L''_1+ \left(1-(\theta_0')^2\right)L_1 = -2\omega_1 L_0''+2\left(\theta_0'\theta_1'+\omega_1(\theta_0')^2\right) L_0, & L_1(0) = 0, \quad L'_1(0) = -L_d,\\
	... & &
	\end{array}
	\end{equation}
	and
	\begin{equation}
	\begin{array}{lll}
	\epsilon^0:& L_0\theta_0''+2 L_0' \theta_0' = 0, & \theta_0(0)=-\alpha, \quad \theta'_0(0) = 0, \\
	\epsilon^1:& L_0\theta''_1+2 L_0' \theta_1' = -L_1 \theta_0''-2 L_1'\theta_0', & \theta_1(0) = 0, \quad \theta'_1(0) = \theta_d,\\
	... & &
	\end{array}
	\end{equation}
	To save space, we have written only the two first orders of $\epsilon$ since further equations complicate its form very quickly. When we obtain these initial approximation we will see that subsequent equations simplify considerably.
	
	Starting from the $\epsilon^0$ equations we can multiply the one for $\theta_0$ by $L_0$, integrate, and obtain the conservation of angular momentum
	\begin{equation}
	L_0^2 \theta_0' = 0,
	\end{equation}
	where we have used the fact that $\theta_0'(0) = 0$. This can only be satisfied if $\theta_0(\tau^+) = -\alpha$. Therefore, the $L_0$ equation yields the solution $L_0(\tau^+) = 1$ and hence, the leading order solutions are constant. The $\epsilon^1$ equations are now simplified 
	\begin{equation}
	L_1'' + L_1 = 0, \quad \theta_1''=0,
	\end{equation}
	which quickly can be solved to obtain $L_1(\tau^+)=-L_d\sin\tau^+$ and $\theta_1(\tau^+) = \theta_d \tau^+$. These simple initial approximations simplify further asymptotic equations. The are the following
	\begin{equation}
	\begin{array}{lll}
	\epsilon^2:& L''_2+L_2 = - 2 \omega_1 L''_1 + \theta_d^2 - \cos\alpha, & L_2(0) = 0, \quad L'_2(0) = L_d \omega_1,\\
	\epsilon^3:& L''_3+L_3 =  & \\
	&-\left(2\omega_2+\omega_1^2\right) L''_1-2\omega_1 L_2''+\theta_d^2 L_1 + 2 \theta_d \left(\theta_2' +2\omega_1 \theta_d\right) - (\sin\alpha) \theta_1, & L_3(0) = 0, \quad L'_3(0) = L_d (\omega_2-\omega_1^2), \\
	... & &
	\end{array}
	\end{equation}
	and
	\begin{equation}
	\begin{array}{lll}
	\epsilon^2:& \theta_2''+2\theta_d L_1' = -\sin\alpha, & \theta_2(0) = 0, \quad \theta_2'(0) = -\theta_d \omega_1,\\
	\epsilon^3:& \theta_3''+L_1\theta_2'' = -2\left(L_1'\theta_2'+\theta_d L_2'\right) + (\cos\alpha) \theta_1+\omega_1\sin\alpha, & \theta_3(0) = 0, \quad \theta'_3(0) = -\theta_d (\omega_2-\omega_1^2), \\
	... & &
	\end{array}
	\end{equation}
	Notice that the $\epsilon^2$ order $L$-equation will not be forced by a resonant term only if we take $\omega_1 = 0$. Similarly, in the next order equation we can eliminate the secular terms when we take
	\begin{equation}
	\omega_2 = -\frac{1}{2}\theta_d^2,
	\end{equation}
	since $L''_1 = - L_1$. Next, solving for $L_2$ yields the (formal) behaviour of the solution up to the second order
	\begin{equation}
	\label{eqn:LApprox}
	\begin{split}
	L(t) \approx \widetilde{L}(t) := 1-\epsilon L_d \sin\tau^+(t)-\epsilon^2 \left(\cos\alpha-\theta_d^2\right)\left(1-\cos\tau^+(t)\right), \\
	\text{where} \quad \tau^+(t)=\widetilde{\omega}(\epsilon)\frac{t}{\epsilon}, \quad \widetilde{\omega}(\epsilon)=1-\frac{\theta_d^2}{2}\epsilon^2, \quad \epsilon\ll 1.
	\end{split}
	\end{equation}
	On the other hand, now we can go back to the $\epsilon^2$ order equation for the angle and solve it immediately to obtain the approximation for $\theta$
	\begin{equation}
	\label{eqn:ThetaApproxTau}
	\theta(t) \approx \widetilde{\theta}(t) := -\alpha + \epsilon \theta_d \tau^+(t) -\frac{1}{2} \epsilon^2 (\sin\alpha) (\tau^+(t))^2 + 2\epsilon^2 L_d\theta_d \left(1-\cos\tau^+(t)\right),
	\end{equation}
	where $\tau^+(t)$ is defined in (\ref{eqn:LApprox}). 
	
	Having obtained the above approximation in the fast $\tau^+$ scale it is interesting to analyse the angle equation in (\ref{eqn:MainEq}) in the slow $t$ scale. Notice that since $L(t)\approx 1$ for $\epsilon \ll 1$ in the first approximation the fast time scale $\tau^+$ enters the equation only through the damping term. In that case we obtain a second order equation with a quickly varying coefficient which could be tackled by the homogenization theory (see for ex. \cite{Pav08}). However, in our case in order to find the leading order of $\theta$ it is more convenient to use the general multiple-scales method. 
	
	We start by assuming that $\theta = \theta(t,\tau)$. Since we are interested only in the leading order form of the asymptotic expansion we do not have to use the strained time scale $\tau^+$ which would introduce higher order terms. As an expansion for $L$ we use 
	\begin{equation}
	L(\tau) = 1 - \epsilon \lambda_1(\tau) - \epsilon^2 \lambda_2(\tau) - ...,
	\end{equation}
	while the angle is expanded as follows
	\begin{equation}
	\theta(t,\tau) = \theta_0(t,\tau) + \epsilon \theta_1(t,\tau) + \epsilon^2 \theta_2(t,\tau)+ ...
	\end{equation}
	The exact form of $\lambda_1$ can be inferred from (\ref{eqn:LApprox}). The initial conditions can be translated into the expansion as
	\begin{equation}
	\label{eqn:ICThetaPert}
	\begin{array}{llll}
	\theta_0(0,0) = -\alpha, & \dot{\theta}_0(0,0) = \theta_d, & \theta'_0(0,0) = 0, &\\
	\theta_i(0,0) = 0, & \dot{\theta}_i(0,0) = 0, & \theta'_i(0,0) = 0, & i > 0,
	\end{array}
	\end{equation}
	where, similarly as above, the dot indicates the derivative with respect to $t$ while prime is the derivative with respect to $\tau$. The equation can now be expanded to yield
	\begin{equation}
	\begin{split}
	&\left(1-\epsilon \lambda_1 - ...\right)\left(\ddot{\theta}_0+\epsilon^{-2} \theta''_0 + 2\epsilon^{-1}\dot{\theta}'_0 + \epsilon \left(\ddot{\theta}_1+\epsilon^{-2}\theta''_1+2\epsilon^{-1}\dot{\theta}'_1\right)+\epsilon^2\left(\ddot{\theta}_2+\epsilon^{-2}\theta''_2+2\epsilon^{-1}\dot{\theta}'_2\right)\right) \\
	& - 2\left(\lambda'_1+\epsilon\lambda_2'+...\right)\left(\dot{\theta}_0+\epsilon^{-1} \theta'_0 + \epsilon \dot{\theta}_1 + \theta'_1 +... \right) = \sin\left(\theta_0+\epsilon\theta_1+\epsilon^2\theta_2+...\right).
	\end{split}
	\end{equation}
	Comparing various powers of $\epsilon$ yields an array
	\begin{equation}
	\begin{array}{ll}
	\epsilon^{-2}:& \theta_0'' = 0, \\
	\epsilon^{-1}:& 2\dot{\theta}'_0  - \lambda_1\theta''_0 + \theta''_1 - 2\lambda_1'\theta_0' = 0,\\
	\epsilon^{0}:& \ddot{\theta}_0 +2\dot{\theta}'_1 - \lambda_1\left(2\dot{\theta}'_0+\theta''_1\right)-\lambda_2 \theta''_0 -2\lambda'_1\left(\dot{\theta}_0+\theta'_1\right) - 2\lambda'_2\theta'_0 + \theta''_2 = \sin\theta_0,\\
	...
	\end{array}
	\end{equation}
	The first equation immediately gives us $\theta_0(t,\tau) = C_0(t) + D_0(t)\tau$. But from the initial conditions (\ref{eqn:ICThetaPert}) we have $D_0(t) = 0$ and hence $\theta'_0 = 0$. As anticipated, the leading order term does not depend on the fast time scale $\tau$. Similarly, for the $\epsilon^{-1}$ equation initial conditions yield $\theta'_1 = 0$ and hence $\theta_1(t,\tau) = C_1(t)$. 
	
	The first non-trivial behaviour comes from the $\epsilon^0$ equation which reduces to
	\begin{equation}
	\ddot{\theta}_0 + \theta''_2 - 2\lambda'_1\dot{\theta}_0 = \sin\theta_0.
	\end{equation}
	Now, since by assumption $t$ and $\tau$ are independent variables, we can average over the fast time scale to extract information only about the slow evolution. Integrating above yields
	\begin{equation}
	\ddot{\theta}_0 + \overline{\theta''_2} - 2\dot{\theta}_0\overline{\lambda'_1} = \sin\theta_0,
	\end{equation}
	where overline denotes the average over $\tau$ changing from $0$ to $2\pi$. Since $\lambda_1$ is $2\pi$-periodic up to $O(\epsilon)$ due to (\ref{eqn:LApprox}), we have $\overline{\lambda'_1} = 0$. Moreover, since $\tau$ enters the equation only through the periodic terms we anticipate that $\theta_2$ is $\tau$ periodic and hence the average of its derivatives vanishes. This leaves
	\begin{equation}
	\label{eqn:Theta0}
	\ddot{\theta}_0 = \sin\theta_0,
	\end{equation}
	which is the leading order equation for the evolution of $\theta_0$. This approximation suffices for our needs and we will not continue the multiple scales analysis. We therefore claim that
	\begin{equation}
	\label{eqn:ThetaApprox}
	\theta(t) \approx \widetilde{\theta_0}(t), \quad \epsilon\ll 1, 
	\end{equation}
	where $\widetilde{\theta_0}$ is given by the solution of (\ref{eqn:Theta0}) (it can be solved analytically in terms of the Jacobi amplitude function $am$). Notice also that if we put $t = \epsilon \tau$ and expand the solution of (\ref{eqn:Theta0}) with respect to $\epsilon\rightarrow 0$ for fixed $\tau$ the first two terms correspond to the nonperiodic part of (\ref{eqn:ThetaApproxTau}) with $\tau^+$ replaced by $\tau$. Moreover, if we make the reasonable small angle assumption $\sin\theta = \theta + O(\theta^3)$ we will have
	\begin{equation}
	\theta(t) \approx -\alpha\cosh t + \theta_d \sinh t, \quad \epsilon\ll 1, \quad \alpha \ll 1.
	\end{equation}
	
	Of course we can continue the multiple scales approach and obtain higher order terms. For this program to be successful, we should also include the $t$-scale expansion of the pendulum length $L$. We will not pursue this topic here since approximations soon become very complicated and are not needed in what follows. 
	
	The above analysis has been intended to be formal yet illustrative to clearly state the possible form of the singular asymptotic expansion and the interplay of multiple time scales. We now proceed to the rigorous proofs of the above results.
	
	\subsection{Rigorous proofs}
	
	First, we need the following result which is a generalization of Gr\"onwall-Bellman's lemma. Its generalized version has been proven in \cite{Has66} but here, for completeness, we include a simplified proof of the case we need. 
	
	\begin{lem}\label{lem:GB}
		Let $\psi=\psi(t)$ and $f=f(t)$ be continuous and positive functions. Assume that the following inequality holds
		\begin{equation}
		\psi(t) \leq f(t) + C\int_0^t (t-s) \psi(s) ds,
		\end{equation}
		for $C>0$. Then
		\begin{equation}
		\psi(t) \leq f(t) + \sqrt{C} \int_0^t \sinh\left(\sqrt{C}(t-s)\right)f(s)ds.
		\end{equation}
		In particular, where $f\equiv D =$const. we have
		\begin{equation}
		\psi(t) \leq D \cosh\left(\sqrt{C}t\right).
		\end{equation}
	\end{lem} 
	\begin{proof}
		First, put $y(t) = \int_0^t (t-s)\psi(s) ds$. Then, it follows that $y'' = y$ and hence from the assumption
		\begin{equation}
		y''(t) - C y(t)\leq f(t).
		\end{equation}
		If we add and subtract $\sqrt{C}y'$ and multiply both sides by $e^{\sqrt{C}t}$ we arrive at
		\begin{equation}
		\left(e^{\sqrt{C}t}y'\right)' - \sqrt{C} \left(e^{\sqrt{C}t} y\right)' \leq e^{\sqrt{C}t} f(t).
		\end{equation}
		That is to say
		\begin{equation}
		y' - \sqrt{C} y \leq e^{-\sqrt{C}t}\int_0^t e^{\sqrt{C}s}f(s)ds,
		\end{equation}
		where we have used the fact that by the definition we have $y(0)=0$ and $y'(0)=0$. Once again, multiplying by the integrating factor $e^{-\sqrt{C}t}$ we have
		\begin{equation}
		\left(e^{-\sqrt{C}t}y\right)' \leq e^{-2\sqrt{C}t}\int_0^t e^{\sqrt{C}s}f(s)ds.
		\end{equation}
		The last integration gives
		\begin{equation}
		y(t) \leq e^{\sqrt{C}t} \int_0^t e^{-2\sqrt{C}u}\left(\int_0^u e^{\sqrt{C}s}f(s)ds\right)du = e^{\sqrt{C}t}\int_0^t \left(\int_s^t e^{-2\sqrt{C}u}du\right)e^{\sqrt{C}s}f(s)ds,
		\end{equation}
		and finally by evaluating the inner integral
		\begin{equation}
		y(t) \leq \frac{1}{\sqrt{C}}\int_0^t \sinh \left(\sqrt{C}(t-s)\right) f(s) ds.
		\end{equation}
		The assertion easily follows by the assumption and definition of $y$. The proof is complete. 
	\end{proof}
	
	We are ready to prove the main result of this paper. 
	\begin{thm}[Fast time scale asymptotics]
		\label{thm:Asym}
		Let $(L(\tau^+),\theta(\tau^+))$ be the solution of (\ref{eqn:LTau+}). Then, the following asymptotic behaviour holds
		\begin{equation}
		|L(\tau^+) - \widetilde{L}(\tau^+)| = O(\epsilon^3), \quad |\theta(\tau^+) - \widetilde{\theta}(\tau^+)|= O(\epsilon^3), \quad \text{as} \quad \epsilon\rightarrow 0^+,
		\end{equation}
		uniformly for $\tau^+\leq T<\infty$, where $\widetilde{L}$ and $\widetilde{\theta}$ are defined in (\ref{eqn:LApprox}) and (\ref{eqn:ThetaApproxTau}). 
	\end{thm}
	\begin{proof}
		We begin by finding the asymptotic behaviour of $L$. Write
		\begin{equation}
		L(t) = \widetilde{L}(t) + \lambda(t),
		\end{equation}
		and change the time scale into the strained fast time $\tau^+ = \widetilde{\omega}(\epsilon) \epsilon^{-1} t$ (with a slight abuse of notation). According to (\ref{eqn:MainEq}) and (\ref{eqn:LApprox}) the equation for the remainder $\lambda$ has the form
		\begin{equation}
		\widetilde{\omega}^2 \lambda'' + \lambda = f(\tau^+,\epsilon) + (\epsilon\dot{\theta}(\epsilon \widetilde{\omega}^{-1}\tau^+))^2 \lambda,
		\end{equation}
		where, as before, the differentiation with respect to $\tau^+$ is denoted with a prime while the derivative with respect to $t$ is denoted with a dot, and
		\begin{equation}
		\label{eqn:f}
		\begin{split}
		f(\tau^+,\epsilon) :&= \epsilon\left(1-\widetilde{\omega}^2 - 
		\epsilon^2 \dot{\theta}(\epsilon \widetilde{\omega}^{-1}\tau^+)^2\right) L_d \sin\tau^+ \\
		&+\epsilon^2 \left[\dot{\theta}(\epsilon \widetilde{\omega}^{-1}\tau^+)^2-\cos \theta(\epsilon \widetilde{\omega}^{-1}\tau^+) + \widetilde{\omega}^2 \left(\cos\alpha-\theta_d^2\right)\cos\tau^+\right. \\
		&\left.+ \left(\cos\alpha-\theta_d^2\right)\left(1-\cos\tau^+\right)- \epsilon^2\left(\cos\alpha-\theta_d^2\right)\left(1-\cos\tau^+\right)\dot{\theta}(\epsilon \widetilde{\omega}^{-1}\tau^+)^2 \right].
		\end{split}
		\end{equation}
		Moreover, the initial conditions for the remainder are zero: $\lambda(0)=\dot{\lambda}(0)=0$. Notice that we have deliberately retained the $t$-derivative of $\theta$ evaluated at a point $\epsilon \widetilde{\omega}^{-1}\tau^+$. 
		
		The equation for $\lambda$ can be easily transformed into an integral equation with the use of the Green's function
		\begin{equation}
		\label{eqn:LambdaIntEq}
		\lambda(\tau^+) = \int_0^{\tau^+} \sin\left(\widetilde{\omega}(\epsilon)^{-1}\left(\tau^+-s\right)\right) f(s,\epsilon) ds + \epsilon^2 \int_0^{\tau^+} \sin\left(\widetilde{\omega}(\epsilon)^{-1}\left(\tau^+-s\right)\right) \dot{\theta}\left(\epsilon \widetilde{\omega}(\epsilon)^{-1}s\right)^2 \lambda(s)ds.
		\end{equation}
		Out claim is that $f(\tau^+,\epsilon)=O(\epsilon^3)$ uniformly with respect to $\tau^+\leq T$. To prove it, first notice that by (\ref{eqn:LApprox}) we have
		\begin{equation}
		1-\widetilde{\omega}^2 - \epsilon^2 \dot{\theta}(\epsilon \widetilde{\omega}^{-1}\tau^+)^2 = \epsilon^2\left(\theta_d^2 - \dot{\theta}(\epsilon \widetilde{\omega}^{-1}\tau^+)^2 - \frac{1}{4}\epsilon^2\theta_d^4\right).
		\end{equation}
		Since, by uniform continuity on compact intervals $\dot{\theta}(\epsilon \widetilde{\omega}^{-1}\tau^+)^2 = \theta_d^2 + O(\epsilon)$, the first term in the definition of $f$, namely (\ref{eqn:f}), is $O(\epsilon^4)$ as $\epsilon\rightarrow 0$. Now, the terms in the bracket in (\ref{eqn:f}) can be written as
		\begin{equation}
		\label{eqn:term}
		\begin{split}
		&\dot{\theta}(\epsilon \widetilde{\omega}^{-1}\tau^+)^2-\theta_d^2+\cos\alpha-\cos \theta(\epsilon \widetilde{\omega}^{-1}\tau^+) + \epsilon^2 \left(\cos\alpha-\theta_d^2\right)\left(\dot{\theta}(\epsilon \widetilde{\omega}^{-1}\tau^+)^2-\theta_d^2\right)\cos\tau^+\\
		&- \epsilon^2\left(\cos\alpha-\theta_d^2\right)\dot{\theta}(\epsilon \widetilde{\omega}^{-1}\tau^+)^2 + \frac{1}{4}\epsilon^4 \theta_d^4 \left(\cos\alpha-\theta_d^2\right)\cos\tau^+.
		\end{split}
		\end{equation}
		Note that $\cos\alpha-\cos \theta(\epsilon \widetilde{\omega}^{-1}\tau^+)=O(\epsilon)$.
		Again, we use uniform continuity of $\dot{\theta}$ and conclude that the terms (\ref{eqn:term}) are $O(\epsilon)$. Combining the previous two estimates we conclude that
		\begin{equation}
		|f(\tau^+,\epsilon)| \leq D \epsilon^3 \text{ uniformly for }\tau^+\leq T \text{ for } \epsilon\rightarrow 0,
		\end{equation}
		for some constant $D>0$. Now, from the integral equation (\ref{eqn:LambdaIntEq}) we can infer, for $C>0$, that
		\begin{equation}
		|\lambda(\tau^+)| \leq D\epsilon^3+C\epsilon^2 \int_0^{\tau^+} \left(\tau^+-s\right)|\lambda(s)|ds,
		\end{equation}
		where we have used the fact that $\dot{\theta}$ is bounded and $|\sin(u-v)|\leq|u-v|$. By Lemma \ref{lem:GB} it follows that
		\begin{equation}
		\label{eqn:LambdaEst}
		|\lambda(\tau^+)| \leq D\epsilon^3 \cosh\left(\sqrt{C}\epsilon \tau^+\right) \leq D\epsilon^3 \cosh\left(\sqrt{C}T  \epsilon \right) = O(\epsilon^3) \quad \text{as} \quad \epsilon\rightarrow 0. 
		\end{equation}
		And the first part of the assertion is proved.
		
		The second part proceeds analogously by writing
		\begin{equation}
		\theta(t) = \widetilde{\theta}(t) + \psi(t),
		\end{equation}
		for $\widetilde{\theta}$ defined in (\ref{eqn:ThetaApproxTau}). In the $\tau^+$ scale the governing equation (\ref{eqn:LTau+}) yields the evolution of the remainder
		\begin{equation}
		\label{eqn:PsiPert}
		L\psi''+2 L' \psi' = g(\tau^+,\epsilon) + \widetilde{\omega}^{-2}\epsilon^2 \sin(\widetilde{\theta}+\psi),
		\end{equation}
		where
		\begin{equation}
		g(\tau^+,\epsilon) := 2\epsilon L' \left(-\theta_d+\epsilon(\sin\alpha) \tau^+ - 2\epsilon L_d\theta_d \sin\tau^+  \right) + \epsilon^2 L \left(\sin\alpha - 2 L_d\theta_d \cos\tau^+\right).
		\end{equation}
		We can transform (\ref{eqn:PsiPert}) into its equivalent integral form by multiplying it by $L$ and integrating twice keeping in mind vanishing initial conditions. Hence, 
		\begin{equation}
		\label{eqn:PsiIntEq}
		\psi(\tau^+) = \int_0^{\tau^+} G(\tau^+,s,\epsilon) \left[g(s,\epsilon) + \widetilde{\omega}^{-2}\epsilon^2 \sin\left(\widetilde{\theta}(s)+\psi(s)\right)\right]ds,
		\end{equation}
		where the kernel is defined by
		\begin{equation}
		\label{eqn:Kernel}
		G(\tau^+,s,\epsilon) = L(s)^{-2} \int_s^{\tau^+} L(u) du.
		\end{equation}
		First, we will extract the meaningful information from $g$. To this end use the first part of the theorem to write $L = 1 - \epsilon L_d\sin\tau^+ + O(\epsilon^2)$ uniformly with respect to $\tau^+\leq T$ as $\epsilon\rightarrow 0$ to obtain
		\begin{equation}
		\begin{split}
		g(\tau^+,\epsilon) &= \epsilon^2\left[ 2 L_d \cos\tau^+ \left(\theta_d-\epsilon(\sin\alpha) \tau^+ + 2\epsilon L_d\theta_d \sin\tau^+  \right) \right.\\
		&\left.+ \left(1-2\epsilon L_d\sin\tau^+\right) \left(\sin\alpha - 2 L_d\theta_d \cos\tau^+\right)\right] + O(\epsilon^3).
		\end{split}
		\end{equation}
		Now, the terms with $\cos\tau^+$ cancel leaving
		\begin{equation}
		g(\tau^+,\epsilon) = \epsilon^2 \sin\alpha + O(\epsilon^3),
		\end{equation}
		since all other terms are uniformly bounded. Whence, the integral equation (\ref{eqn:PsiIntEq}) becomes
		\begin{equation}
		\label{eqn:PsiIntEq2}
		\psi(\tau^+) = \epsilon^2\int_0^{\tau^+} G(\tau^+,s,\epsilon) \left[ \widetilde{\omega}^{-2} \sin\left(\widetilde{\theta}(s)+\psi(s)\right)+\sin\alpha + O(\epsilon)\right]ds.
		\end{equation}
		Next, by a simple estimate we have
		\begin{equation}
		\left|\widetilde{\omega}^{-2} \sin\left(\widetilde{\theta}(s)+\psi(s)\right)+\sin\alpha\right| \leq \widetilde{\omega}^{-2} \left|\sin\left(\widetilde{\theta}(s)+\psi(s)\right)+\sin\alpha\right| + \left|\sin\alpha||\widetilde{\omega}^{-2}-1\right|,
		\end{equation}
		which, by the fact that $\widetilde{\theta} +\alpha = O(\epsilon)$, implies
		\begin{equation}
		\label{eqn:SinEst}
		\left|\widetilde{\omega}^{-2} \sin\left(\widetilde{\theta}(s)+\psi(s)\right)+\sin\alpha\right| \leq \left|\widetilde{\omega}^{-2}\right| |\psi(s)| + O(\epsilon).
		\end{equation}
		Moreover, due to asymptotic expansion of $L$ the kernel can be written as
		\begin{equation}
		\label{eqn:KernelAsym}
		G(\tau^+,s,\epsilon) = \tau^+-s + O(\epsilon),
		\end{equation}
		uniformly for $\tau^+\leq T$ as $\epsilon\rightarrow 0$. Therefore, combining (\ref{eqn:SinEst}) and (\ref{eqn:KernelAsym}) with (\ref{eqn:PsiIntEq2}) we arrive at
		\begin{equation}
		|\psi(\tau^+)| \leq E \epsilon^3 + F \epsilon^2 \int_0^{\tau^+} (\tau^+-s) |\psi(s)| ds,
		\end{equation}
		for some constants $E, F > 0$ and we have used the fact that all the $O(\epsilon^3)$ terms are uniform with respect to $\tau^+\leq T$. Invoking Lemma \ref{lem:GB} finally yields
		\begin{equation}
		\label{eqn:PsiEst}
		|\psi(\tau^+)| \leq E \epsilon^3 \cosh\left(\sqrt{F}\epsilon \tau^+\right) \leq E \epsilon^3 \cosh\left(\sqrt{F}\epsilon T\right) = O(\epsilon^3) \quad \text{as} \quad \epsilon\rightarrow 0.
		\end{equation}
		This ends the proof.
	\end{proof}
	
	From the above proof we can immediately spot a place when the usual trade-off between the order of approximation and interval length can be made. Notice that in both final estimates (\ref{eqn:LambdaEst}) and (\ref{eqn:PsiEst}) we could even allow for $\tau^+\leq \epsilon^{-1} T$ and still uniformly bound the exponential term. Therefore, the expansions should be valid on longer $\epsilon$-expanding intervals. As we will see this is only partially true and we loose one order of convergence. 
	\begin{cor}\label{cor:Asym}
		Let $(L(t),\theta(t))$ be the solution of (\ref{eqn:LTau+}). Then, the following asymptotic behaviour holds
		\begin{equation}
		|L(\tau^+) - \widetilde{L}(\tau^+)| = O(\epsilon^2), \quad |\theta(\tau^+) - \widetilde{\theta}(\tau^+)|= O(\epsilon^2), \quad \text{as} \quad \epsilon\rightarrow 0^+,
		\end{equation}
		on an expanding interval $\tau^+\leq \epsilon^{-1}T<\infty$, where $\widetilde{L}$ and $\widetilde{\theta}$ are defined in (\ref{eqn:LApprox}) and (\ref{eqn:ThetaApproxTau}). 
	\end{cor}
	\begin{proof}
		It suffices to reanalyse the proof of Theorem \ref{thm:Asym}. When estimating the size of (\ref{eqn:f}) we used the assumption that $\epsilon\widetilde{\omega}\tau^+\rightarrow 0$ when $\epsilon\rightarrow 0$ uniformly with respect to $\tau^+$ on compact intervals. Now, since $\epsilon\tau^+$ stays bounded in that limit we can only conclude that $|f(\tau^+,\epsilon)|\leq C\epsilon^2$. The final estimate (\ref{eqn:LambdaEst}) now reads
		\begin{equation}
		|\lambda(\tau^+)| \leq D\epsilon^2 \cosh\left(\sqrt{C}\epsilon \tau^+\right) \leq D\epsilon^2 \cosh\left(\sqrt{C}T\epsilon\right) = O(\epsilon^2) \quad \text{as} \quad \epsilon\rightarrow 0. 
		\end{equation}
		The proof for $\theta$ has to be changed exactly in the same way. We obtain $g(\tau^+,\epsilon) = \epsilon^2 \sin\alpha + O(\epsilon^2)$ and continue the reasoning accordingly. 
	\end{proof}
	We thus see the interplay between the two time scales. The pendulum length $L$ oscillates on the $\tau^+$ scale with the period modulated by the evolution of $\theta$ on the slow scale. Since $\theta(\widetilde{\omega}^{-1}\epsilon\tau^+)$ is uniformly continuous only on compact subsets of $\tau^+$, the asymptotic expansion looses one order to account for that on $\epsilon$-expanding $\tau^+$ intervals. It is also a very well known fact that without coupling between $L$ and $\theta$ the Poincar\'e-Lindstedt series would approximate the solutions with full order on $\epsilon$-expanding intervals. 
	
	Having in mind the above discussion we can prove the leading-order asymptotic expansion of $\theta$ on the slow $t$ scale. 
	\begin{thm}[Slow time scale asymptotics]
		Let $\theta=\theta(t)$ be solution of (\ref{eqn:MainEq}). Then,
		\begin{equation}
		|\theta(t)-\widetilde{\theta}_0(t)| = O(\epsilon) \quad \text{as} \quad \epsilon\rightarrow 0,
		\end{equation}
		uniformly for $t\leq T$, where $\widetilde{\theta}_0$ is defined in (\ref{eqn:ThetaApprox}).  
	\end{thm}
	\begin{proof}
		The integral equation for $\theta$ can be obtained by multiplying the first equation in (\ref{eqn:MainEq}) by $L$ which brings up the angular momentum
		\begin{equation}
		\frac{d}{dt}\left(L^2\dot{\theta}\right) = L\sin\theta.
		\end{equation}
		The above can be integrated twice and manipulated to yield
		\begin{equation}
		\theta(t) = -\alpha + \theta_d \int_0^t L(s)^{-2} ds + \int_0^t G(t,s,\epsilon)\sin\theta(s)ds,
		\end{equation}
		where the kernel is given by (\ref{eqn:Kernel}). Now, write 
		\begin{equation}
		\theta(t) = \widetilde{\theta}_0(t) + \psi(t),
		\end{equation}
		where $\widetilde{\theta}_0(t)$ is a solution of (\ref{eqn:Theta0}) with the original initial conditions (\ref{eqn:ICPolar}) and hence $\psi(0)=\dot{\psi}(0)=0$. Then, by the above argument we have
		\begin{equation}
		\psi(t) = \int_0^t G(t,s,\epsilon)\left(\sin\left(\theta_0(s)+\psi(s)\right)-L(s)\ddot{\theta}_0(s)-2\dot{L}(s)\dot{\theta}_0(s)\right)ds.
		\end{equation}
		Due to Corollary \ref{cor:Asym} we have $L(t) = 1-\epsilon L_d \sin\left(\widetilde{\omega}\epsilon^{-1}t\right) + \lambda(t)$, where $\lambda = O(\epsilon^2)$ as $\epsilon\rightarrow 0$ uniformly for $t\leq T$ since a compact interval for $t$ is $\epsilon$-expanding for $\tau^+$. Hence
		\begin{equation}
		\begin{split}
		\psi(t) &= \int_0^t G(t,s,\epsilon) \left(\sin\left(\theta_0(s)+\psi(s)\right)-\sin\theta_0(s)\right) ds \\
		&+ \int_0^t G(t,s,\epsilon) \left(\epsilon L_d\sin\left(\widetilde{\omega}(\epsilon)\frac{s}{\epsilon}\right)- \lambda(s)\right) \sin\theta_0(s) ds \\
		&+ 2 \int_0^t G(t,s,\epsilon) \dot{\theta}_0(s) \left(L_d \widetilde{\omega}(\epsilon) \cos\left(\widetilde{\omega}(\epsilon)\frac{s}{\epsilon}\right) - \dot{\lambda}(s) \right)ds =: g_1(t) + g_2(t) + g_3(t),
		\end{split}
		\end{equation}
		where we used the fact that $\ddot{\theta}_0 = \sin\theta_0$. 
		
		The $\psi$ term is associated with $g_1$ and we can estimate it as follows
		\begin{equation}
		|g_1(t)| \leq E_1\int_0^t (t-s) |\psi(s)| ds,
		\end{equation}
		where $E_1>0$. Further, for some constants $E_2, F_2 > 0$, the next term is simply
		\begin{equation}
		|g_2(t)| \leq F_2\int_0^t (t-s) \left(\epsilon + \epsilon^2\right)ds \leq E_2 \epsilon,
		\end{equation}
		by the fact that $t\leq T$. Lastly, we have $\dot{\lambda} = O(\epsilon)$ and by integration by parts we obtain
		\begin{equation}
		\begin{split}
		&\int_0^t G(t,s,\epsilon) \dot{\theta}_0(s) \widetilde{\omega}(\epsilon) \cos\left(\widetilde{\omega}(\epsilon)\frac{s}{\epsilon}\right)ds =\left[ \epsilon G(t,s,\epsilon) \dot{\theta}_0(s) \sin\left(\widetilde{\omega}(\epsilon)\frac{s}{\epsilon}\right)\right]_{s=0}^t \\
		&- \epsilon \int_0^t \frac{d}{ds} \left( G(t,s,\epsilon) \dot{\theta}_0(s)\right) \sin\left(\widetilde{\omega}(\epsilon)\frac{s}{\epsilon}\right)ds.
		\end{split}
		\end{equation}
		The term in the brackets vanishes because $G(t,t,\epsilon)=0$ which leaves us with
		\begin{equation}
		|g_3(t)|\leq E_3 \epsilon,
		\end{equation}
		where $E_3>0$. Combining our results for $g_i$, $i=1, 2, 3$, we now have
		\begin{equation}
		|\psi(t)|\leq E\epsilon + F \int_0^t (t-s)|\psi(s)|ds,
		\end{equation}
		for some constants $E,F>0$. Invoking Lemma \ref{lem:GB} yields
		\begin{equation}
		|\psi(t)|\leq E\epsilon \cosh \left(\sqrt{F} t\right) \leq E\epsilon \cosh \left(\sqrt{F} T\right) = O(\epsilon) \quad \text{as} \quad \epsilon\rightarrow 0.
		\end{equation}
		The proof is complete.
	\end{proof}
	
	We have thus found the exact asymptotic expansion of $L$ and $\theta$ on two time scales. Now, we will proceed to applying these results to solving the boundary value problems stated at the beginning of this paper. 
	
	\section{Boundary value problem}
	
	Armed with above results we will proceed to reanalyse Problem \ref{prob:main}. First, we will prove that we can always find its unique solution at least for sufficiently small initial angles $\alpha$. 
	
	\begin{thm}
		\label{thm:BVP}
		There exists a number $\alpha_0>0$ such that Problem \ref{prob:main} has a unique solution for $|\alpha|<\alpha_0$. Moreover, 
		\begin{equation}
		\label{eqn:TStar}
		t^* = \epsilon\pi + O(\epsilon^3) \quad \text{as} \quad \epsilon\rightarrow 0,
		\end{equation}
		and
		\begin{equation}
		\label{eqn:KApprox}
		K^* \approx \widetilde{K}^* := \left(\frac{\pi \theta_d}{2\alpha}\right)^2 = \left(\frac{\pi }{2\alpha}\left(U\cos\alpha-V\sin\alpha\right)\right)^2.
		\end{equation}
	\end{thm}
	\begin{proof}
		We will work on the fast $\tau^+$ scale. Since $L = \widetilde{L}+O(\epsilon^3)$ the time $\tau^*$ of the first return of $L$ to its initial condition satisfies
		\begin{equation}
		\label{eqn:TauStar}
		L_d \sin\tau^* + \epsilon \left(\cos\alpha-\theta_d^2\right)\left(1-\cos\tau^*\right) + O(\epsilon^2) = 0,
		\end{equation}
		as $\epsilon\rightarrow 0$. We can see that to the leading order $\tau^*\approx \pi$. This observation can be made more accurate. Let $\mu$ be the solution of the following equation
		\begin{equation}
		L_d \sin\mu + \epsilon \left(\cos\alpha-\theta_d^2\right)\left(1-\cos\mu\right) = 0,
		\end{equation}
		then, by classical theory we will have $\tau^* = \mu + O(\epsilon^2)$ as $\epsilon\rightarrow 0$. Since $\mu>0$ we have $\sin\mu = \sqrt{1-\cos^2\mu}$ and the above is a quadratic equation in $\cos\mu$. The solution is
		\begin{equation}
		\cos\mu = -1 + \frac{2\epsilon^2\left(\cos\alpha-\theta_d^2\right)^2}{L_d^2+\epsilon^2\left(\cos\alpha-\theta_d^2\right)^2}.
		\end{equation}
		Hence, the skipped terms in (\ref{eqn:TauStar}) are of the same order as the difference $\tau^*-\mu$. Therefore,
		\begin{equation}
		\tau^* = \pi + O(\epsilon^2) \quad \text{as} \quad \epsilon\rightarrow 0,
		\end{equation}
		and (\ref{eqn:TStar}) follows. 
		
		Now, since we know $\tau^*$ we require that $\theta(\tau^*) = \alpha$. Multiplying $\theta$ equation (\ref{eqn:LTau+}) by $L$, integrating twice, and using the kernel (\ref{eqn:Kernel}) we obtain
		\begin{equation}
		\alpha=\theta(\tau^*(\epsilon)) = -\alpha+\epsilon \widetilde{\omega}^{-1}\theta_d\int_0^{\tau^*(\epsilon)} L(s,\epsilon)^{-2}ds + \epsilon^2 \widetilde{\omega}^{-2} \int_0^{\tau^*(\epsilon)} G(\tau^*(\epsilon),s,\epsilon) \sin\theta(s,\epsilon) ds.
		\end{equation}
		We have to show that there exists a number $\epsilon(\alpha)$ for which the above has a solution. To this end define
		\begin{equation}
		F(\alpha,\epsilon) = -2\alpha+\epsilon\widetilde{\omega}^{-1}\theta_d\int_0^{\tau^*(\epsilon)} L(s,\epsilon)^{-2}ds + \epsilon^2 \widetilde{\omega}^{-2} \int_0^{\tau^*(\epsilon)} G(\tau^*(\epsilon),s,\epsilon) \sin\theta(s,\epsilon) ds.
		\end{equation} 
		Observe that $F(0,0) = 0$ and
		\begin{equation}
		\begin{split}
		\frac{\partial F}{\partial \epsilon} (\alpha,\epsilon )&= \widetilde{\omega}^{-1}\theta_d\int_0^{\tau^*(\epsilon)} L(s,\epsilon)^{-2}ds \\
		&+ \epsilon \left(\frac{\partial}{\partial\epsilon}\left(\widetilde{\omega}^{-1}\theta_d\int_0^{\tau^*(\epsilon)} L(s,\epsilon)^{-2}ds \right)+ 2 \widetilde{\omega}^{-2} \int_0^{\tau^*(\epsilon)} G(\tau^*(\epsilon),s,\epsilon) \sin\theta(s,\epsilon) ds \right) \\
		&+ \epsilon^2 \frac{\partial}{\partial\epsilon}\left(\widetilde{\omega}^{-2} \int_0^{\tau^*(\epsilon)} G(\tau^*(\epsilon),s,\epsilon) \sin\theta(s,\epsilon) ds\right).
		\end{split}
		\end{equation}
		When we put $\epsilon = 0$ only the first term above survives and hence
		\begin{equation}
		\frac{\partial F}{\partial \epsilon} (\alpha,0) = \theta_d \pi \neq 0.
		\end{equation}
		Therefore, by the Implicit Function Theorem it follows that there exists a number $\alpha_0$ and a function $\epsilon:(-\alpha_0,\alpha_0)\rightarrow \mathbb{R}$, such that $F(\alpha_0,\epsilon(\alpha_0))=0$. The boundary value problem has thus a unique solution. 
		
		The last part of the proof is to find an approximation to the solution. Since $\tau^* \approx \pi + O(\epsilon^2)$ we can use $\widetilde{\theta}$ to determine $\epsilon(\alpha)$ when we truncate $O(\epsilon^2)$ terms. We have
		\begin{equation}
		\alpha = -\alpha + \epsilon\theta_d \pi.
		\end{equation}
		Solving and remembering that $\epsilon = K^{-1/2}$ yields (\ref{eqn:KApprox}) and the proof is complete. 
	\end{proof}
	
	After proving our results it is required to find the conditions for $\widetilde{L}$, $\widetilde{\theta}$ and $\widetilde{K}^*$ to be good approximations of the corresponding quantities. First of all, the perturbation expansion of $L$ and $\theta$ have been obtained under the assumption that $K\rightarrow\infty$ (i.e. $\epsilon\rightarrow 0$) with all other parameters fixed. We can \textit{a posteriori} verify the assumptions on these. This can be done by noting that the subsequent terms in the expansions (\ref{eqn:LApprox}) and (\ref{eqn:ThetaApproxTau}) have to be of higher order. That is to say, we require that $\epsilon^2 \theta_d^2 \ll \epsilon L_d$ and $\epsilon^2 L_d \theta_d \ll \epsilon\theta_d$. Expressing this in terms of $K$ yields the consistency condition
	\begin{equation}
	\label{eqn:ConsistencyCondition}
	\frac{1}{\sqrt{K}} \theta_d^2 \ll L_d \ll \sqrt{K}, \quad K\gg 1.
	\end{equation}
	For example, the above is satisfied if $L_d,\theta_d = O(1)$. Moreover, from (\ref{eqn:KApprox}) we read that the requirement of $K\gg 1$ simultaneously keeping the conditions (\ref{eqn:ConsistencyCondition}) satisfied forces $\alpha \ll 1$. Hence, the angle should be small what is also consistent with the data. Further, if we go back to (\ref{eqn:ThetaLdUV}) and express $L_d$ and $\theta_d$ in terms of the Froude numbers $U$ and $V$, we see that $L_d, \theta_d \approx U$ since $V\leq U$ for fixed $\alpha$. Whence, (\ref{eqn:ConsistencyCondition}) reduce to
	\begin{equation}
	U \ll \sqrt{K}, \quad K \gg 1.
	\end{equation}
	Moreover, from the above we can infer about the validity of the approximation (\ref{eqn:KApprox}) which can be plugged into above to conclude that
	\begin{equation}
		\frac{2 U}{\pi} \alpha \ll V \ll \frac{\pi U}{2\alpha},
	\end{equation}
	which is consistent with the small angle assumption. Referring to Tab. \ref{tab:Parameters} we see that our asymptotic approximations are valid in the realistic regime of parameters. 
	
	In \cite{Mc90} several approximations of $K^*$ have been proposed. Authors claimed that for slower velocities, the dependence of $K^*$ on $U$ should be quadratic while for faster velocities, linear. In both of these cases Authors gave very complex fitted empirical formulas which closely reproduced the numerical results. Our approximation (\ref{eqn:KApprox}) gives a systematic explanation of the leading order behaviour of $K^*$ for small angles. It can also be treated as an approximation of the quadratic part of dependence on $U$. Note however, that some other components of the velocity might be missing and finding them is a subject of our future work. Furthermore, in the cited work Authors introduce the so-called effective vertical stiffness $K_{vert}$ which reduce to $K$ when the subject is required to jump vertically (i.e. $\alpha=0$ and $V=0$). Authors heuristically derive that
	\begin{equation}
	K_{vert}\approx \left(Q\frac{\pi U}{\sin\alpha}\right)^2,
	\end{equation}
	where $Q$ is an unknown constant dependent on the contact time. Notice the similarity to our systematically devised result (\ref{eqn:KApprox}). This suggest that for small velocities and angles, the two stiffness parameters behave in a similar fashion. 
	
	We will illustrate our results with several numerical simulations. On Fig. \ref{fig:Error} we can see an exemplary verification of Theorem \ref{thm:Asym}. Error is calculated on the fast $\tau^+$ scale by choosing a compact interval $[0,\pi]$ and comparing solutions of (\ref{eqn:LTau+}) with their approximations (\ref{eqn:LApprox}) and (\ref{eqn:ThetaApproxTau}) at $\tau^+=\pi$ for the worst case. Observe that on the log-log scale the plots become parallel to the superimposed $K^{-3/2}$ line indicating the correct order of convergence. Note also that we have used the original variables $K$, $U$ and $V$. 
	
	\begin{figure}
		\centering
		\includegraphics{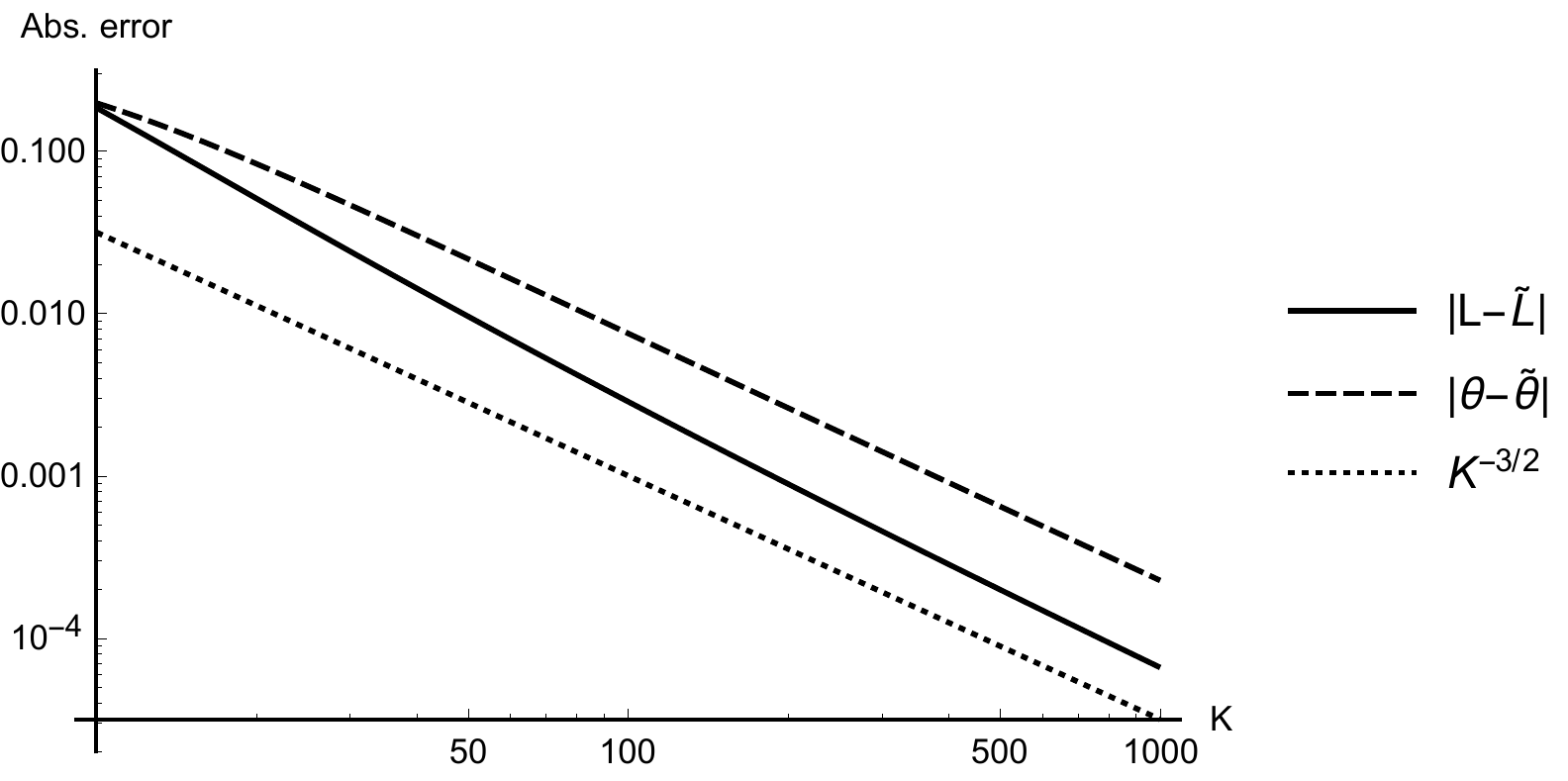}
		\caption{Absolute error of the approximations (\ref{eqn:LApprox}) and (\ref{eqn:ThetaApproxTau}) of the solutions to (\ref{eqn:LTau+}) plotted on the log-log scale. The line $y = K^{-3/2}$ is superimposed for comparison. Here, $\alpha=0.4$, $U=1$, $V=0.1$, and $\tau^+\in[0,\pi]$. }
		\label{fig:Error}
	\end{figure}

	The next example concerns the validity of (\ref{eqn:KApprox}) as the approximation to the solution of Problem \ref{prob:main}. Since, as we noted above, the natural assumption for its accuracy is $\alpha\ll 1$, we compare $K^*$ with $\widetilde{K}^*$ for different values of the angle. Results are given on Fig. \ref{fig:KK}. We can see that both values are close to each other and their ration converge to $1$ when $\alpha\rightarrow 0^+$. Note, however, that although decently accurate, $\widetilde{K}^*$ is only the leading order approximation for $K^*$. Finding the subsequent corrections is the aim of our future work. 
	
	\begin{figure}
		\centering
		\includegraphics[width=0.49\textwidth]{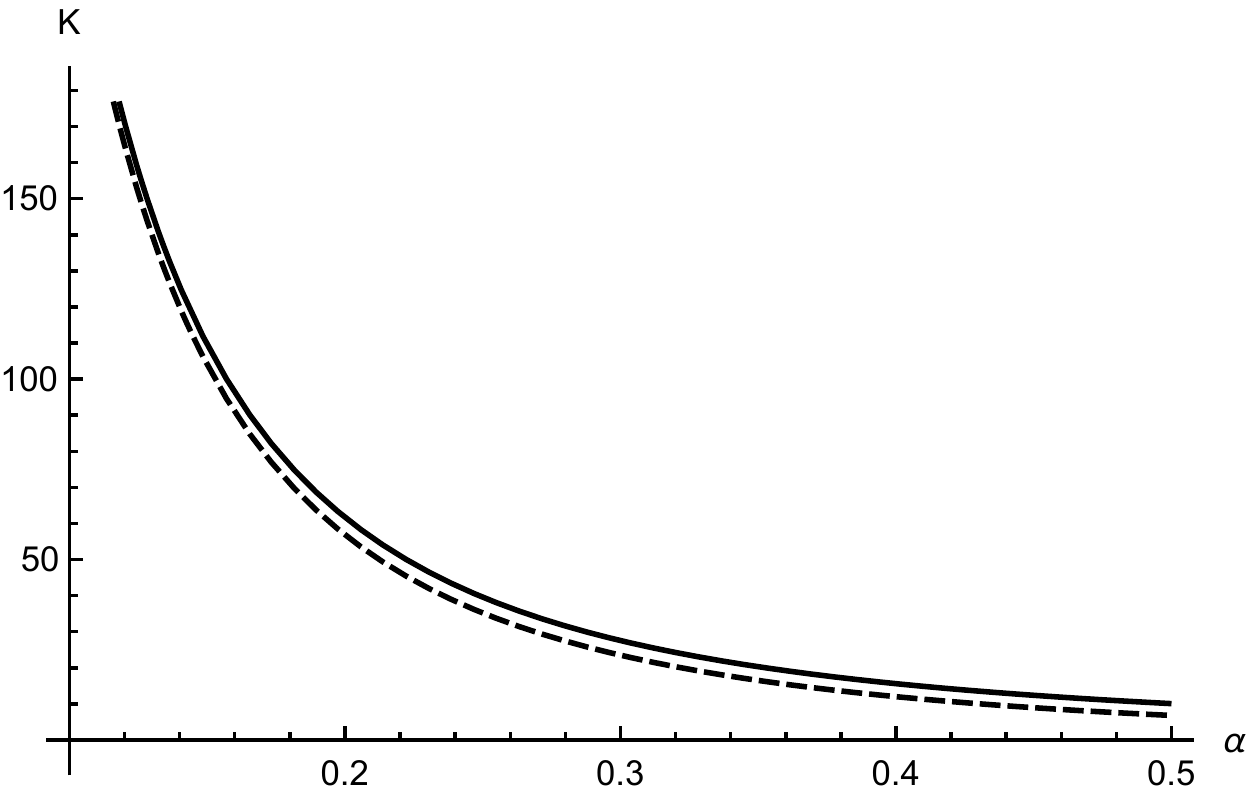}
		\includegraphics[width=0.49\textwidth]{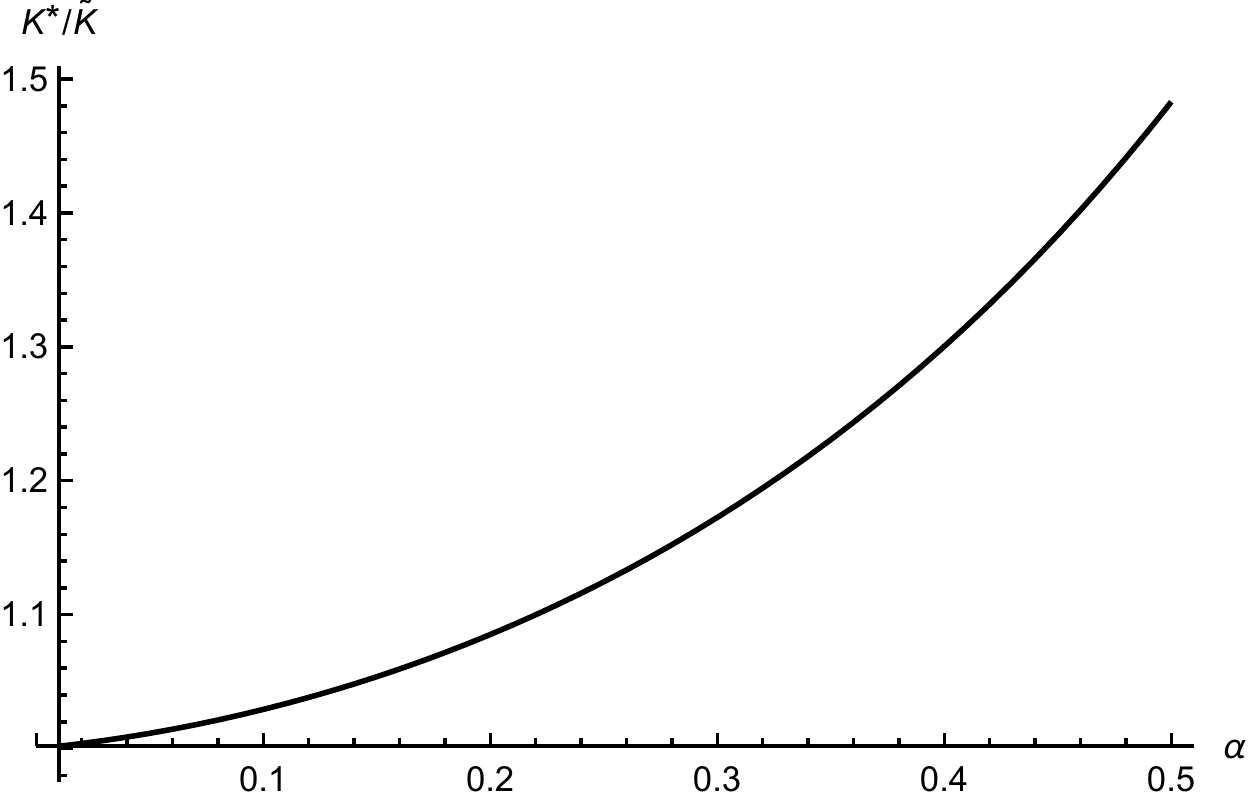}
		\caption{Comparison between the numerically calculated value of $K^*$ and its approximation $\widetilde{K}^*$ calculated from (\ref{eqn:KApprox}) for varying $\alpha$. On the left: $K^*$ (solid line) and $\widetilde{K}^*$ (dashed line). On the right: ratio of $K^*$ to $\widetilde{K}$. Here, $U=1$ and $V=0.1$.}
		\label{fig:KK}
	\end{figure}
	
	\section{Conclusion}
	We have solved a nonlinear boundary value problem that is very natural to modelling legged locomotion. It appeared that the equations live on multiple time scales, however, the problem had its solution on the fast scale $\tau^+$. Having proved the validity of asymptotic expansions we had used them in applying the Implicit Function Theorem to grant the existence and uniqueness of solution to Problem \ref{prob:main}. It is also worth to mention that the approximation of the stiffness (\ref{eqn:KApprox}) is consistent with all of the previously speculated features of the numerical solution. We have thus justified several claims about its behaviour for a realistic regime of parameters.
	
	The further work is will be based on considering expansions for larger velocities and finding the subsequent terms in the expansion for $K^*$ with respect to small $\alpha$. As was noted in \cite{Mc90} the stiffness starts to depend linearly for large values of $U$. We also plan to justify this claim analytically. 
	

\end{document}